\documentclass[reqno]{amsart}
\usepackage{amssymb,latexsym}
\usepackage[utf8]{inputenc}
\usepackage[english]{babel}
\usepackage{amsmath,amsthm}
\usepackage{enumerate}
\usepackage{enumitem}
\usepackage{graphicx}
\usepackage{mathtools}
\usepackage{calc}

\DeclarePairedDelimiter{\floor}{\lfloor}{\rfloor}
\usepackage{pgf,tikz}
\usetikzlibrary{arrows}
\usepackage{pgfplots}
\usepackage[hyperfootnotes=false, colorlinks, linkcolor={blue}, citecolor={magenta}, filecolor={blue}, urlcolor={blue}]{hyperref}

\usepackage{pgffor}
\usepackage{tikz-cd}
\usepackage{booktabs}

\usepackage{enumerate}

\usepackage{siunitx}

\newcommand*{\specialcell}[2][b]{%
	\begin{tabular}[#1]{@{}c@{}}#2\end{tabular}%
}
\newcommand*{\specialcellbold}[2][b]{%
	\bfseries
	\sisetup{text-rm = \bfseries}%
	\begin{tabular}[#1]{@{}c@{}}#2\end{tabular}%
}

\makeatletter
\def\paragraph{\@startsection{paragraph}{4}%
	\z@\z@{-\fontdimen2\font}%
	{\normalfont\bfseries}}
\makeatother
\theoremstyle{plain}
\newtheorem{theorem}{Theorem}
\newtheorem{corollary}{Corollary}

\newtheorem{lemma}{Lemma}

\newtheorem{Remark}{Remark}

\theoremstyle{definition}
\newtheorem{definition}{Definition}

\theoremstyle{remark}

\numberwithin{equation}{section}

\DeclareMathOperator{\End}{End}

\newcommand{\RR}{\mathbb{R}}
\newcommand{\ZZ}{\mathbb{Z}}
\newcommand{\FF}{\mathbb{F}}
\newcommand{\Or}[1]{\mathcal{O}r(#1)}

\begin{document} %
\title{Orchards in elliptic curves over finite  fields}

\author{R. Padmanabhan}
\address{R. Padmanabhan\\
   Department of Mathematics\\
   University of Manitoba\\
   Winnipeg, Manitoba  R3T 2N2, 
   Canada}
\email{padman@cc.umanitoba.ca}

\author {Alok Shukla }
\address{Alok Shukla\\
    Mathematical and Physical Sciences division\\
    School of Arts and Sciences \\
    Ahmedabad University \\
    Central Campus, Navrangpura, Ahmedabad 380009, India}
\email{Alok.Shukla@ahduni.edu.in}




\begin{abstract}
	Consider a set of $ n $ points on a plane.	A line containing exactly $ 3 $ out of the $ n $ points is called a $ 3 $-rich line. 
	The classical orchard problem asks for a configuration of the $ n $ points on the plane that maximizes the number of $ 3 $-rich lines. In this note, using the group law in elliptic curves over finite fields, we exhibit several (infinitely many) group models for orchards wherein the number of $ 3 $-rich lines agrees with the expected number given by Green-Tao (or, Burr, Gr\"unbaum and Sloane) formula for the maximum number of lines. We also show, using elliptic curves over finite fields, that there exist infinitely many point-line configurations with the number of $ 3 $-rich lines exceeding the expected number given by Green-Tao formula by two, and this is the only other optimal possibility besides the case when the number of $ 3 $-rich lines agrees with the Green-Tao formula. 
\end{abstract}


\maketitle

\section{Introduction} \label{Sec:Intro}
\textit{Orchard problem} has a rich history going back to 1821. Jackson posed this problem in a poetic style in his book, aptly entitled, ``Rational Amusement for Winter Evenings \ldots'' \cite{jackson1821rational}. This problem was also posed by Sylvester in 1868 \cite{sylvester1868mathematical}. \textit{The orchard problem} is to plant $ n $ trees in an orchard maximizing the number of possible rows which contain exactly three trees. Of course, trees can be considered as points on a plane and rows as lines. We say that a line is $ 3 $-rich if it contains exactly three points out of the $ n $ specified points. A $ (n,t) $-arrangement is a set of $ n $ points and $ t $ $ 3 $-rich lines in the real projective plane. The classical orchard problem is to find an arrangement in the real projective plane with greatest number of $ 3 $-rich lines $ t $ for each given value of $ n $. We say that $ (n,t) $ is optimal and $ \Or{n} =t $, if $ t $ is the solution of the orchard problem for $ n $ points. Clearly, for $ 3 $ (also, $ 4 $) points there is only one $ 3 $-rich line possible, so arrangement $ (3,1) $ (also, $ (4,1) $) is optimal, and $ \Or{3} = \Or{4} =1 $. Fig.~\ref{Fig:ExampleOrchard} shows optimal arrangements $ (5,2),(6,4) $, $ (7, 6)$ and $ (8,7) $.
	
	Sylvester proved that $ \Or{n} \geq \floor{\frac{1}{6}(n-1)(n-2)} $, where, as noted earlier, $ \Or{n} $ is the maximum number of $ 3 $-rich lines for $ n $ points.  \textup{(}Here, and in what follows, $\floor{x} $ denotes the integer part of $x$.\textup{)} Burr, Gr\"unbaum and Sloane \cite{MR337659}, discussed this problem in a paper in $ 1973 $, and proved that 
	\begin{equation}\label{Eq:Burr}
	\Or{n} \geq \floor {n (n-3)/6} + 1.
	\end{equation} 
\begin{figure}[ht]
	\includegraphics[]{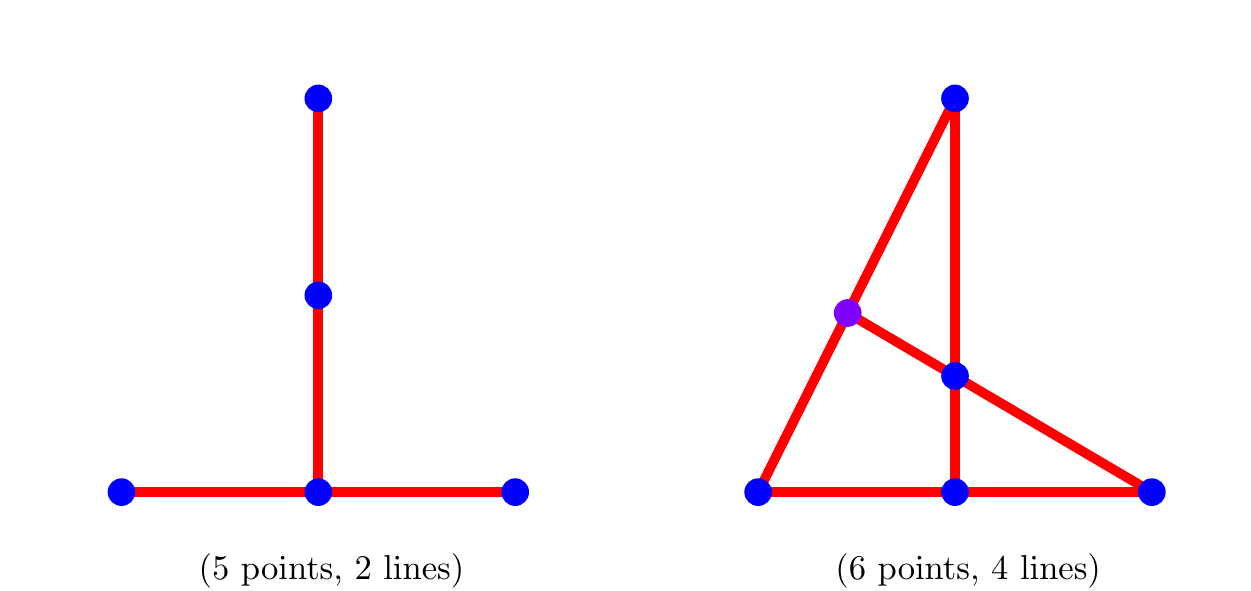}
	\includegraphics[]{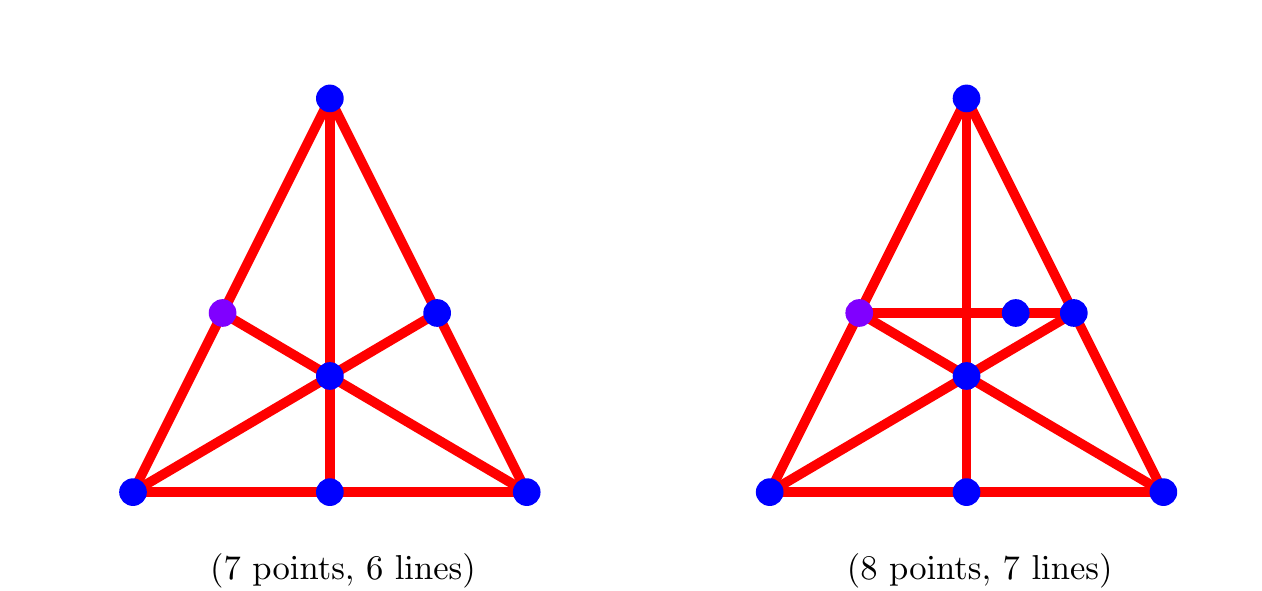}
	\caption{Examples of optimal $ (p,t) $ arrangements for the orchard problem for  $ 5 \leq p \leq 8 $.}\label{Fig:ExampleOrchard}
\end{figure}

More recently, in 2013 Green and Tao solved the orchard problem for all but finitely many cases. More precisely, they have solved the orchard problem for such cases wherein the number of points $ n $ is greater than some sufficiently large constant. 
\begin{theorem}[Green-Tao, \cite{green2013sets}]\label{Theorem:GreenTao}
	Suppose that $P$ is a finite set of $n$ points in the plane. Suppose that $n \geq n_0$ for some sufficiently large absolute constant $n_0$. Then there are no more than $\lfloor n (n-3)/6\rfloor + 1$ lines that are \emph{$3$-rich}, that is they contain precisely $3$ points of $P$.  
\end{theorem}
In conjunction with \eqref{Eq:Burr}, Green-Tao theorem shows that for all $n \geq n_0$,  
\begin{equation} \label{Eq:Green-Tao-Bound}
\Or {n} = \floor {n (n-3)/6} + 1.
\end{equation}

For brevity, we will call the bound given by Eq.~\eqref{Eq:Green-Tao-Bound}, ``Green-Tao" bound in the sequel. Although, it should more appropriately be addressed as ``Burr, Gr\"unbaum, Sloane, Green, Tao" bound owing to the result of Burr, Gr\"unbaum and Sloane given in Eq.~\eqref{Eq:Burr}.

Since, every line cuts a cubic in three points, and four points on a cubic will never be collinear, cubic curves defined over real or complex numbers provide a natural tool to construct orchards. This has been done before, for example, by Burr, Gr\"unbaum and Sloane \cite{MR337659} and Green-Tao \cite{green2013sets} in obtaining their results described earlier. 

In this note, we give analogs of the classical orchard problem, by changing the `ground plane' of orchards from a real projective plane to a projective plane defined over a finite field.\\

\noindent \textbf{Orchard problem over finite fields:} Consider a set of  $ N $ points in a projective plane defined over a finite field. Find an optimal configuration of the given $ N $ points, such that the number of $ 3 $-rich lines is maximum. \\

\noindent \textbf{Notation:} In the sequel, by a slight abuse of notation, the number of $ 3 $-rich lines in an optimal configuration for the orchard problem over finite fields will also be denoted by $ \Or{N} $. The finite field with $ q $ elements will be denoted by $ \FF_q $, where $ q = p^n $ for some prime $ p $ and some positive integer $ n $. For an elliptic curve $ E $ given by a Weierstrass equation $ \displaystyle y^2 + a_1xy +a_3y = x^3 +a_2x^2 +a_4x +a_6 $ defined over a field $ K $, $ E(K) $ will denote the set $ \displaystyle \{(x,y) \in K^2 \ | \ y^2 + a_1xy +a_3y = x^3 +a_2x^2 +a_4x +a_6    \} \cup \{O\} $. \\

Certain elliptic curves defined over finite fields will be used to obtain concrete group models for the orchards $ (N,t) $.
In these group models, three points $ P,Q,R $ form a $ 3 $-rich line, if and only if, $ P+Q+R = O $ in the underlying group structure. This enables one to count the number of lines and we give infinite families of examples using both ordinary and supersingular elliptic curves over finite fields wherein the Green-Tao bounds (see \eqref{Eq:Green-Tao-Bound}) are attained. Also, we show that there is just one other possibility, which also occurs infinitely many times, that the Green-Tao bounds are exceeded by two $ 3 $-rich lines. 

Next we consider a numerical example to illustrate how the orchard problem can be treated on an elliptic curve over a finite field.

\begin{description}
	\item[Example (6 points, 4 lines)]  Consider the elliptic curve $ E $ given by a minimal Weierstrass equation $ y^2 = x^3 + 3 $  over $ \FF_5 $. It can be verified that 
$ \#E(\FF_5) = 6 $ and $ E(\FF_5) $ contains the following points:
\begin{align*}
 A &= (1, 2), \quad B = (1, 8),\quad C = (2, 1), \\
 D &= (2, 4), \quad E = (3, 0), \quad  O = (0,1,0)  \text{ (the point at infinity).}
\end{align*}

 The following four $ 3 $-rich lines could be formed using the points in $ E(\FF_5) $ (see Fig.~\ref{Fig:6points4lines}): 
 \begin{align*}
 \l_1: \quad   \text{containing} \quad  \{O, A, B\},\\
 \l_2: \quad   \text{containing} \quad  \{O, C, D\},\\
 \l_3: \quad   \text{containing} \quad  \{A, C, E\},\\
 \l_4: \quad   \text{containing} \quad  \{B, D, E\}.
 \end{align*}
 It can be verified that $ E(\FF_5) \cong \ZZ_6 $ and $ A = (1,2) $ is a generator of the group. It can be verified that all the points above lie on the elliptic curve 
 $ y^2 + 5xy -15 y = x^3 -10 x^2 +35x -42 $, which is $  y^2 = x^3 + 3 \mod 5 $. 
 
\begin{figure}
	\begin{center}
		\includegraphics[scale=0.7]{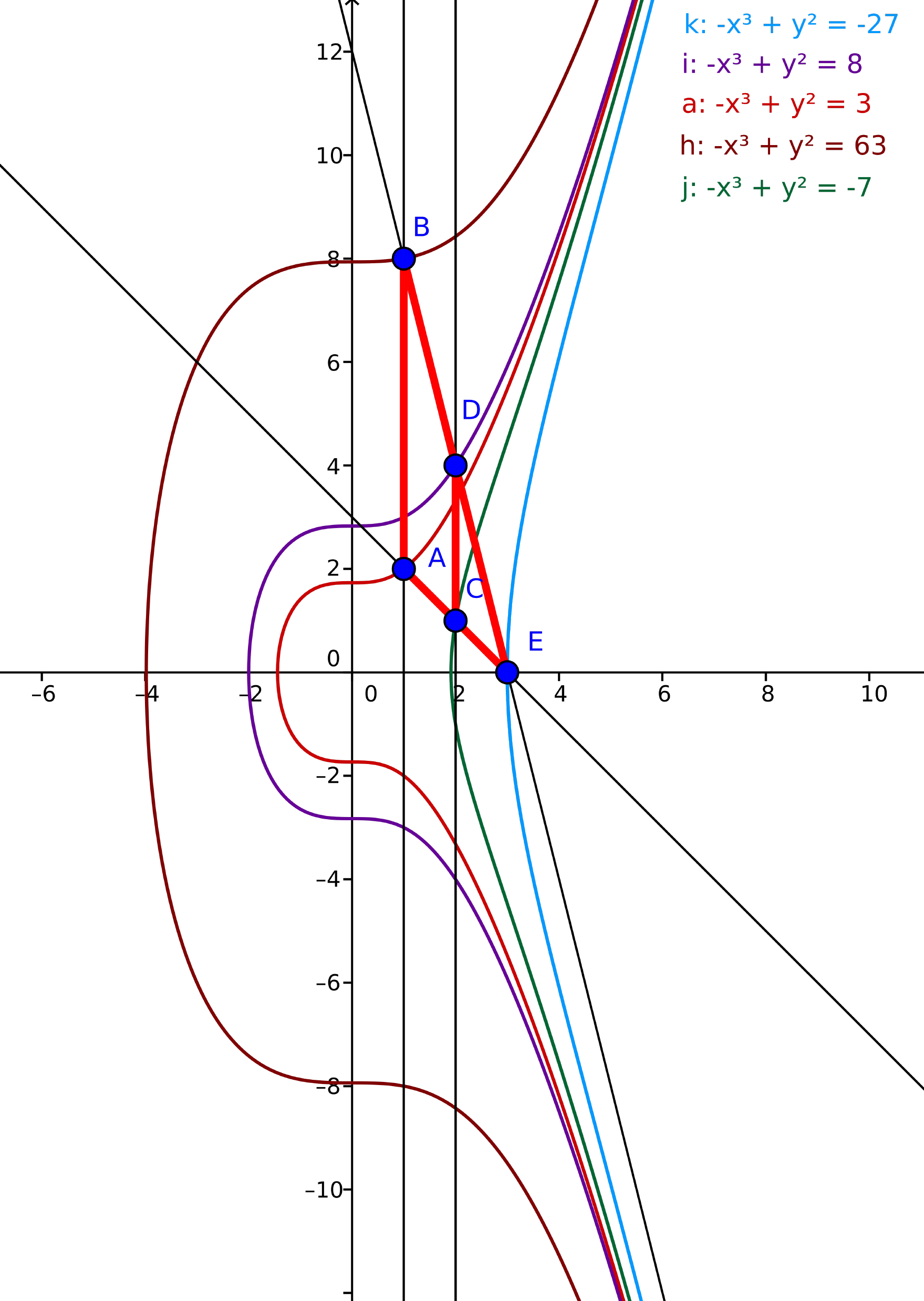}
	\end{center}
	\caption{Orchard configuration with (6 points, 4 lines) using elliptic curve given by a minimal Weierstrass equation $ y^2 = x^3 + 3 $  over $ \FF_5 $. Note that all the five elliptic curves shown in the figure reduce to $ y^2 = x^3 + 3 $ modulo $ 5 $. }\label{Fig:6points4lines}
\end{figure}
\end{description}

Now we give a brief organization of the paper. In the next section, i.e. Sect.~\ref{Sect:KnownResults}, some important results on elliptic curves over finite fields will be reviewed. 
In Sect.~\ref{Sect:main} our main results are proved. Finally, in Sect.~\ref{Sect:Examples} we discuss some numerical examples and the question of realizability of our solutions on a real projective plane.

\section{Elliptic curves over finite fields} \label{Sect:KnownResults}
 The main purpose of this section is to collect some known structural results on elliptic curves over finite fields for later uses. Readers are referred to standard textbooks, such as \cite{MR2514094} or \cite{MR2404461}, for details on the material related to this section.

An \textit{elliptic curve} over a field $ K $ is a smooth cubic curve of genus one with a specified basepoint $ O \in E(K) $. It is a consequence of Riemann-Roch theorem that every elliptic curve $ E/K $ can be given by an explicit Weierstrass equation of the form $ \displaystyle y^2 + a_1xy +a_3y = x^3 +a_2x^2 +a_4x +a_6 $, with $ a_i \in K $, and the extra point $ [0,1,0] $ at infinity. If characteristic of $ K $ is not $ 2 $ or $ 3 $ one can always write equation of an elliptic curve in the following minimal Weierstrass form
\begin{equation}
E: y^2 = x^3 + Ax +B,
\end{equation}
with the discriminant $ \Delta = -16 (4A^3 + 27 B^2)  \neq 0 $.  Here the condition $ \Delta \neq 0  $ ensures that the elliptic curve is non-singular with no-repeated roots (equivalently, the partial derivatives do not vanish simultaneously for any point on $E$), which geometrically means the curve has no cusps or nodes (self-intersections). 
For an elliptic curve $ E $ there is an important invariant, called $ j $-invariant which is defined as
$ \displaystyle j(E) =  1728 \dfrac{4A^3}{4 A^3 + 27 B^2}$. Two elliptic curves defined over a field $ K $ and having the same $ j $-invariant are isomorphic over $ \bar{K} $, where $ \bar{K} $ is the algebraic closure of the ground field $ K $. 
An elliptic curve is equipped with an abelian group structure and the group law is such that three points $ P,Q $  and $ R $ in $ E(K) $ are collinear if and only if $ P + Q + R = O $, with $ O $ being the identity element of the group, i.e., the point at infinity $ (0,1,0) $. \\

An \emph{isogeny } between $E_1/K$ and $E_2/K$ is a non-constant morphism $\phi \colon E_1 \to E_2$ defined over $K$ such that $\phi(O_{E_1})= O_{E_2}$. 
(i.e., $ \phi $ is a morphism of curves given by rational functions with coefficients in $K$, which respects identity.)
Elliptic curves $E_1$ and $E_2$ are called isogenous if there exists an isogeny $\phi \colon E_1 \to E_2$. Endomorphism ring of $ E/K $, $ \End_{K}(E) $ (or $ \End(E) $, to ease notation) consists of all the isogenies $ \phi \colon E \to E $, with natural addition and multiplication given by composition. 
For any integer $ m $, there is a natural definition of an isogeny \emph{multiplicaiton by $ m $} given as  $ [m]: E \to E $, with $ [m]P = mP$. The kernel of the isogeny $ [m] $ is denoted by 
$ E[m] $ (i.e., $ E[m] = \{ P \in E  \ | \ m P = 0\} $), and it is called the \emph{$ m $-torsion subgroup} of $ E $. 

Mordell proved that the group of rational points  $ E(Q) $ is finitely generated. Moreover, over a finite field $ E(\FF_q)$ is a finite abelian group of order $ \# E(\FF_q) $, where a bound for $ \# E(\FF_q) $ is given by Hasse's theorem.
\begin{theorem}[Hasse]
 
\begin{equation}
	\text{Let }\, \# E(\FF_q) = q+1 -t. \, \text{Then }  |t|  \leq 2 \sqrt{q}.
\end{equation}
\end{theorem}
Next, we note a formula to determine the number $ \# E(\FF_q) $. 
\begin{equation}
\# E(\FF_q) =  q + 1 +\sum_{{x \in \FF_q}}  \left(\frac{x^3 + Ax + B}{\FF_q} \right),
\end{equation}
where $\left( \frac{x}{\FF_q}\right) $ is the generalized Legendre symbol which takes the value $ 1 $ if $ x $ is a square in $ \FF_q^{\times} $, and $ -1 $ if $ x $ is not a square in $ \FF_q^{\times} $, otherwise, it takes the value $ 0 $ if $ x=0 $.

\begin{theorem}[Deuring, \cite{MR5125}] \label{Theorem:Deuring}
	Let $ E $ be an elliptic curve over the finite field $ \FF_q $. For each integer $ m \geq 1 $, let $ \phi_m \colon E \to  E^{(p^m)}  $ and $ \hat{\phi}_{m} \colon E^{(p^m)} \to E $
be the $ p^m $-Frobenius map and its dual respectively. Then the following are equivalent.
\begin{enumerate}[]
	\item $ E[p^m] = 0 $ for one(all) $ m \geq 1 $. 
	\item $ \End(E) $ is an order in a quaternion algebra.
	\item $ \hat{\phi}_{m} $ is (purely) inseparable for one (all) $ m \geq 1 $.
	\item The map $ [p] \colon E \to E  $  is purely inseparable and $ j(E) \in \FF_{p^2} $.
\end{enumerate} 	
\end{theorem}
\begin{definition}
	An elliptic curve $ E/ \FF_q $ is called supersingular if it satisfies any of the equivalent conditions given in Theorem~\ref{Theorem:Deuring}. Otherwise, it is called an ordinary elliptic curve.
\end{definition}
Another characterization of supersingular curves is that a curve is supersingular if and only if $ t \equiv 0 \mod p $, where $\# E(\FF_q) = q+1 -t$.

The next result describes all the possible values of $ \#E(q^n) $ that one can get by varying $ E$ over all the elliptic curves over $ \FF_q $.
\begin{theorem}[Schoof, \cite{MR914657}] \label{Theorem:Schoof}
	Let $ q =p^n $. There exists an elliptic curve $ E/\FF_q $ of order $ q+1 -t $ if and only if one of the following condition holds:
	\begin{enumerate}
		\item $ t \not \equiv 0 \mod p $, and $ t^2 \leq 4q $.
		\item $ n $ is odd and one of the following holds:
		\begin{itemize}
			\item $ t= 0  $.
			\item $ t^2 = 2 q $, and $ p=2 $.
			\item $ t^2 = 3 q $, and $ p=3 $.
		\end{itemize}
		\item $ n $ is even and one of the following holds:
	\begin{itemize}
		\item $ t^2 = 4q $.
		\item  $ t^2 =q  $ and $ p \not \equiv 1 \mod 3 $.
		\item $ t=0 $, and $ p \equiv 1 \mod 4 $.
	\end{itemize}
	\end{enumerate}
	
\end{theorem}
\begin{corollary}[of Theorem~\ref{Theorem:Schoof} ] \label{Cor:Schoof}
	Let $ E  $ be as in Theorem~\ref{Theorem:Schoof}. Then $ E $ is supersingular if and only if $ t^2 = 0$, $ q $, $ 2q $, $ 3q $, or $4q$.
\end{corollary}

\begin{theorem}[R\"{u}ck, \cite{MR890272}] \label{Theorem:Ruck}
	Let $ N = q+1 -t$ such that it occurs as an order of an elliptic curve over $ \FF_q $. Let $ N=  p^e n_1 n_2 $ with $ p \nmid n_1 n_2 $, and $ n_1 | n_2 $ (possibly $ n_1 =1 $). There is an elliptic curve $ E/\FF_q $ such that 
	\begin{equation}
	E(\FF_q) \cong {\ZZ_{p^e}} \, \oplus \,  {\ZZ_{n_1}} \,  \oplus \, {\ZZ_{n_2}}
	\end{equation}
	if and only if 
	\begin{enumerate}
		\item $ n_1 = n_2 $ in the case $ n $ is even and $ t^2 = 4 q $ of Theorem~\ref{Theorem:Schoof}.
		\item $ n_1  |  q- 1 $ in all other cases of Theorem~\ref{Theorem:Schoof}.
	\end{enumerate}
\end{theorem}
In view of  Corollary~\ref{Cor:Schoof} the following theorem characterizes the possible group structures of supersingular elliptic curves $ E(\FF_q) $.
\begin{theorem}[Schoof, \cite{MR914657}] \label{Theorem:SchoofSupersingular}
	Let $ \#E(\FF_q) = q+1 -t$. 
	\begin{enumerate}
		\item If $ t^2 = q $, $ 2q $, or $ 3q $, then $ E(\FF_q) $ is cyclic.
		\item If $ t^2 = 4q  $, then $ E(\FF_q) $ is of the form of $ \ZZ_m  \oplus \ZZ_m    $ with $ m= \sqrt{q} -1  $ if $ t = 2 \sqrt{q} $, otherwise if $ t=- 2\sqrt{q} $ 
		then $ m=  \sqrt{q} + 1  $. 
		\item If $ t=0 $ and $ q \not \equiv 3 \mod 4 $, then $ E(\FF_q) $ is cyclic. If $ t=0 $ and $ q \equiv 3 \mod 4  $, then either $ E(\FF_q) $ is cyclic, or
		 $ E(\FF_q) \cong \, \ZZ_m  \oplus \, \ZZ_2 $, with $ m =\frac{q+1}{2} $. 
	\end{enumerate}
\end{theorem} 
The following theorem gives a simple method for determining whether an elliptic curve is supersingular.
\begin{theorem}[Deuring, \cite{MR5125}] \label{Theorem:DeuringCriteriaSupersingular}
	Let $ p >2 $. The elliptic curve $ E $ over $ \FF_q $, with Weierstrass equation $ y^2 =f(x) $, is supersingular if and only if the coefficient of  $ x^{p-1} $ in $ f(x)^{\frac{p-1}{2}} $ is $ 0 $.
\end{theorem}

Next, for a finite field with even characteristic we have the following results taken from \cite{menezes1990isomorphism} and listed in Table~\ref{Table:odd} and Table~\ref{Table:even}. For $ q=2^n $, a representative curve $ E/\FF_q $ from each of the isomorphism classes of supersingular curve over $ \FF_q $ is listed. The order and the group structure of $ E(\FF_q) $ is also listed along with the curve for both the case when $ n $ is odd (see Table~\ref{Table:odd})) and  $ n $ even (see Table ~\ref{Table:even}).

\begin{table}[htbp]
	\centering
	\begin{tabular}{l*{4}{S[table-format = 2.1]}}
		\toprule
		\textbf{Curve over $ 2^n $}                                    &	
		\textbf{n} &
		{\specialcellbold{Order  }}                 &
		{\specialcellbold{Group}}         & \\
		\midrule 
		\midrule
		$ y^2 + y = x^3 $       &  \text{odd}  & \text{$ q+1 $} &  {cyclic} \\
		\midrule 
		$ y^2 + y = x^3 + x $       &     {\specialcellbold{ $ n \equiv 1,7 \mod 8 $ \\ $ n \equiv 3,5 \mod 8 $}} &  {\specialcell{ $ q+1 + \sqrt{2 q} $ \\ $ q+1 - \sqrt{2 q} $}}  &   {\specialcell{ cyclic \\  cyclic}} \\	
		\midrule
		$ y^2 + y = x^3 + x +1 $       &     {\specialcellbold{ $ n \equiv 1,7 \mod 8 $ \\ $ n \equiv 3,5 \mod 8 $}} & {\specialcell{ $ q+1 - \sqrt{2 q} $ \\ $ q+1 + \sqrt{2 q} $}}  &   {\specialcell{ cyclic \\  cyclic}} \\
			\bottomrule
	\end{tabular} \caption{Elliptic curves over $ \FF_{2^n} $, where $ n $ is odd. } \label{Table:odd}
\end{table}

 \begin{table}[htbp]
 	\centering
 	\begin{tabular}{l*{4}{S[table-format = 2.1]}}
 		\toprule
 		\textbf{Curve over $ 2^n $}                                    &	
 		\textbf{n} &
 		{\specialcellbold{Order  }}                 &
 		{\specialcellbold{Group}}         & \\
 		\midrule 
 		\midrule
 		$ y^2 + y = x^3  + \delta x$       &  \text{even}  & \text{$ q+1 $} &  {cyclic} \\
 		\midrule 
 		$ y^2 + \gamma y = x^3  $       &     {\specialcellbold{ $ n \equiv 0 \mod 4 $ \\ $ n \equiv 2 \mod 4 $}} &  {\specialcell{ $ q+1 + \sqrt{ q} $ \\ $ q+1 - \sqrt{q} $}}  &   {\specialcell{ cyclic \\  cyclic}} \\	
 		\midrule
 		$ y^2 + \gamma y = x^3  + \alpha $       &     {\specialcellbold{ $ n \equiv 0 \mod 4 $ \\ $ n \equiv 2 \mod 4 $}} &  {\specialcell{ $ q+1 - \sqrt{ q} $ \\ $ q+1 + \sqrt{q} $}}  &   {\specialcell{ cyclic \\  cyclic}} \\	 
 		\midrule
 		$ y^2 + \gamma^2 y = x^3  $       &     {\specialcellbold{ $ n \equiv 0 \mod 4 $ \\ $ n \equiv 2 \mod 4 $}} &  {\specialcell{ $ q+1 + \sqrt{ q} $ \\ $ q+1 - \sqrt{ q} $}}  &   {\specialcell{ cyclic \\  cyclic}} \\	
 		\midrule
 		$ y^2 + \gamma^2 y = x^3 + \beta $       &     {\specialcellbold{ $ n \equiv 0 \mod 4 $ \\ $ n \equiv 2 \mod 4 $}} &  {\specialcell{ $ q+1 - \sqrt{ q} $ \\ $ q+1 + \sqrt{ q} $}}  &   {\specialcell{ cyclic \\  cyclic}} \\	
 		\midrule
 		$ y^2 +  y = x^3  $       &     {\specialcellbold{ $ n \equiv 0 \mod 4 $ \\ $ n \equiv 2 \mod 4 $}} &  {\specialcell{ $ q+1 - \sqrt{2 q} $ \\ $ q+1 + \sqrt{2 q} $}}  &   {\specialcell{ $ \ZZ_{\sqrt{q}-1} \oplus \ZZ_{\sqrt{q}-1} $ \\  $ \ZZ_{\sqrt{q}+1} \oplus \ZZ_{\sqrt{q}+1} $}} \\	
 		\midrule
 		$ y^2 +  y = x^3 + \omega $       &     {\specialcellbold{ $ n \equiv 0 \mod 4 $ \\ $ n \equiv 2 \mod 4 $}} &  {\specialcell{ $ q+1 + \sqrt{2 q} $ \\ $ q+1 - \sqrt{2 q} $}}  &   {\specialcell{ $ \ZZ_{\sqrt{q}+1} \oplus \ZZ_{\sqrt{q}+1} $ \\  $ \ZZ_{\sqrt{q}-1} \oplus \ZZ_{\sqrt{q}-1} $}} \\	
 			\midrule		
 		\bottomrule
 	\end{tabular} \caption{Elliptic curves over $ \FF_{2^n} $, where $ n $ is even. } \label{Table:even}
 \end{table}

 In order to use supersingular elliptic curves and their group structure to solve the orchard problem on a projective plane over $ \FF_q $, we consider some representative supersingular elliptic curves. The results of the following lemma are well-known.
 \begin{lemma} \label{Lemma:qoddsupersingular}
 	\leavevmode
 	\begin{enumerate} 
 		\item Let $ q $ be odd and $ q \equiv 2 \mod 3 $. Let the elliptic curve $ E \colon y^2 = x^3 +b $ be defined over $ \FF_q $, with $ b \in \FF_q^{\times} $. Then $ E(\FF_q) \cong \ZZ_{q+1}$.
 		\item Let  $ q \equiv 3 \mod 4$. Let the elliptic curve $ E \colon y^2 = x^3 - x $ be defined over $ \FF_q $. Then $ E(\FF_q) \cong \ZZ_{\frac{q+1}{2}}  \, \oplus \, \ZZ_2  $.
 		\item Let  $ q \equiv 3 \mod 4$. Let the elliptic curve $ E \colon y^2 = x^3 + x $ be defined over $ \FF_q $. Then $ E(\FF_q) \cong \ZZ_{q+1} $.
 	\end{enumerate}
 \end{lemma}
 \begin{proof}
\leavevmode
	\begin{enumerate}
		\item See Prop.~$ 4.33 $, \cite{MR2404461} or Example $ 4.4$,~\cite{MR2514094}) and Example $ 2.17$, \cite{MR1700718}.
		\item See Example $ 2.18$, \cite{MR1700718}.
		\item See Example $ 2.18$, \cite{MR1700718}.
	\end{enumerate}
 \end{proof}
 
As a side remark, we note that recently supersingular elliptic curves have found applications in cryptography.  For example, the supersingular elliptic curve $ y^2 = x^3 + 7 $ is used in Bitcoin's public-key cryptography.

\section{Orchard problem using elliptic curves over finite fields}\label{Sect:main}

Next we prove a lemma which will be essential for proving our main result.
\begin{lemma} \label{Lemma:MainGroupSolutionCount}
	Let $ G $ be a finite abelian group, such that $ G \cong \, \ZZ_{n_1} \, \oplus \ZZ_{n_2} \, \oplus \ZZ_{n_3} \, \oplus \cdots \ZZ_{n_k}$,
	with $ n_i | n_{i+1} $ for $ i =1,2,3,\cdots k-1 $. If $ 3| n_k$, then let $ j $ with $ 1 \leq j \leq k $ be the smallest index such that $ 3| n_j $, i.e., $ j $ is such that $ 3 | n_j $ but $ 3 \nmid n_{j-1} $ (with $ n_0 =1 $). Let 
	\begin{align}
	\Psi(G) = 
	\begin{cases}
0 \qquad &\text{if } 3 \nmid n_k,  \\
k-j  \qquad &\text{if } 3|n_j, 3 \nmid n_{j-1}.	
	\end{cases}
	\end{align}
Then the number of distinct solutions (where order does not matter) of $ x + y + z = 0 $ in $ G $  is 
\begin{equation}
\frac{1}{6}\left(\left(\prod_{i=1}^{k} n_i^2\right)  - 3 \left(\prod_{i=1}^{k} n_i\right)  + 2 (3)^{\Psi(G)}\right).
\end{equation}
\end{lemma}
\begin{proof}	
Let $ x = (\alpha_1,\alpha_2, \cdots \alpha_k) $, $ y = (\beta_1,\beta_2, \cdots \beta_k) $, and $ z = (\gamma_1,\gamma_2, \cdots \gamma_k) $, with $ \alpha_i, \beta_i, \gamma_i \in \ZZ_{n_i} $. Now $ x + y + z = 0 $ in $\ZZ_{n_1} \, \oplus \ZZ_{n_2} \, \oplus \ZZ_{n_3} \, \oplus \cdots \ZZ_{n_k} $ if and only if 
 \begin{equation} \label{Eq:congruences}
 \alpha_i + \beta_i + \gamma_i = 0 \qquad \text{in }\, \ZZ_{n_i}, \quad \text{for } i=1,2,3, \cdots k.
 \end{equation}
In \eqref{Eq:congruences} $ \alpha_i $ and $ \beta_i $ both can take $ n_i $ different values and corresponding to each of the pairs $ (\alpha_i,\beta_i) $ there is a unique $ \gamma_i $, giving $ n_i^2 $ solutions to $ \alpha_i + \beta_i + \gamma_i = 0 $, including solutions with $ \alpha_i $, $ \beta_i $ and $ \gamma_i $ non-distinct. Varying $ i $ from $ 1 $ to $ k $ yields $ \prod_{i=1}^{k} n_i^2 $ such solutions. Next in order to get distinct solutions, $ \prod_{i=1}^{k} 3 n_i $ needs to be subtracted, since for each of the cases $ \alpha_i = \beta_i $, $ \beta_i = \gamma_i $ and $ \gamma_i = \alpha_i $, for $ i=1 $ to $ k $, there are $ \prod_{i=1}^{k} n_i  $ solutions; and to this twice the number of solution when $ \alpha_i = \beta_i = \gamma_i $, for $ i=1 $ to $ k $, should be added. 

Next we compute the number of solutions in case when $ \alpha_i = \beta_i = \gamma_i $, for $ i=1 $ to $ k $.  If $ \alpha_i = \beta_i = \gamma_i $ and $ 3 \nmid n_i $,  then there is a unique solution to $ \alpha_i + \beta_i + \gamma_i =0$, namely $ \alpha_i = \beta_i = \gamma_i = 0 $; otherwise if $ 3 | n_i $, then there are $ 3 $ different solutions to $ \alpha_i + \beta_i + \gamma_i =0$, namely $ \alpha_i =\beta_i = \gamma_i = 0, \frac{n_i}{3} $ and $ \frac{2n_i}{3} $. Using this observation it is easy to see that as $ i $ varies form $ 1 $ to $ k $, there are total $ 3^{\Psi(G)} $ such solutions. 

Therefore, we obtain total $$ \displaystyle \left(\left(\prod_{i=1}^{k} n_i^2\right)  - 3 \left(\prod_{i=1}^{k} n_i\right)  + 2 (3)^{\Psi(G)}\right) $$ solutions, which should be divided by $ 3!=6 $ to obtain the total number of distinct solutions to \eqref{Eq:congruences} where order does not matter.
\end{proof}
As immediate corollaries to Lemma~\ref{Lemma:MainGroupSolutionCount}, we obtain the following results.

\begin{corollary} \label{Cor:n1n2}
	Then the number of distinct solutions (up to ordering)  of $ x + y + z = 0 $ in 	$ \ZZ_{n_1} \oplus \ZZ_{n_2} $, with $ n_1 | n_2 $,	is  
	\begin{align*}
	\begin{cases}
	\frac{n_1^2 n_2^2 - 3n_1 n_2 + 18}{6} \quad  &\text{if }\,  3 | n_1, \\
	\frac{n_1^2 n_2^2 - 3n_1 n_2+ 2}{6} \quad  &\text{if }\, 3 \nmid n_1, \\
	\frac{n_1^2 n_2^2 - 3n_1 n_2+ 6}{6} \quad  &\text{if }\, 3 \nmid n_1, 3 |n_2, \\
	\end{cases} = \,	\begin{cases}
	\floor {\frac{N (N-3)}{6}} + 3 \qquad  &\text{if }\,  3 | N , \\
	\floor {\frac{N (N-3)}{6}} + 1 \qquad  &\text{if }\,  3 \nmid N,
	\end{cases}
	\end{align*}
	where $ N = {n_1n_2} $.
\end{corollary}

\begin{corollary} \label{Cor:mm}
	Then the number of distinct solutions (up to ordering)  of $ x + y + z = 0 $ in 	$ \ZZ_m \oplus \ZZ_m $	is  
	\begin{align*}
	\begin{cases}
	\frac{m^4 - 3m^2 + 18}{6} \quad  &\text{if }\,  3 | m, \\
	\frac{m^4 - 3m^2 + 2}{6} \quad  &\text{if }\, 3 \nmid m, \\
	\end{cases} = \,	\begin{cases}
	\floor {\frac{N (N-3)}{6}} + 3 \qquad  &\text{if }\,  3 | N , \\
	\floor {\frac{N (N-3)}{6}} + 1 \qquad  &\text{if }\,  3 \nmid N,
	\end{cases}
	\end{align*}
	where $ N = m^2 $.
\end{corollary}

\begin{corollary} \label{Cor:N}
	Then the number of distinct solutions (up to ordering)  of $ x + y + z = 0 $ in  $ \ZZ_N $	is  
	\begin{align*}
	\begin{cases}
	\frac{N^2 - 3N + 6}{6} \quad  &\text{if }\,  3 | N, \\
	\frac{N^2 - 3N + 2}{6} \quad  &\text{if }\, 3 \nmid N, \\
	\end{cases} = \,	\begin{cases}
	\floor {\frac{N (N-3)}{6}} + 1 .
	\end{cases}
	\end{align*}
\end{corollary}

\begin{theorem} \label{Theorem:qodd}
		Assume $ \FF_q $ to be a finite field of odd characteristic, with $ q =p^n $. 
		\begin{enumerate}
			\item Let $ q \equiv 3 \mod 4 $. There exist  point-line arrangements $ (N, \floor {\frac{N (N-3)}{6}} + 1 )$, with $ N =q+1 $, in the projective plane over the finite field $ \FF_q$ with group models $ \ZZ_{q+1} $ and $ \ZZ_{\frac{q+1}{2}}  \, \oplus \, \ZZ_2 $. 
			\item Let $ n $ be odd and $ q \not \equiv 3 \mod 4 $, or let $ n $ be even and $ p \not \equiv 3 \mod 4 $.  There exits a point-line arrangement $ (N, \floor {\frac{N (N-3)}{6}} + 1 )$, with $ N =q+1 $, in the projective plane over the finite field $ \FF_q$ with the group model $ \ZZ_{q+1} $.
				\end{enumerate}	
		In all the cases above
		\begin{equation}
		\Or {N} \geq \floor {\frac{N (N-3)}{6}} + 1,
		\end{equation}
		with $ N=q+1 $.
			 
\end{theorem}

\begin{proof}
	\leavevmode
	\begin{enumerate}
		\item Let  $ q \equiv 3 \mod 4$. Let  $ E_1 \colon y^2 = x^3 - x $ and $ E_2 \colon y^2 = x^3 + x $ be elliptic curves defined over $ \FF_q $. Then from Lemma~\ref{Lemma:qoddsupersingular} we have $ E_1(\FF_q) \cong \ZZ_{\frac{q+1}{2}}  \, \oplus \, \ZZ_2  $, and  $ E_2(\FF_q) \cong \ZZ_{q+1} $.
       
        First, we consider $ q+1  $ points on $ E_1(\FF_q) $. Let $ m = \frac{q+1}{2} $.  It is given that $ q \equiv 3 \mod 4 $, so $ 2 |m $. Since, $ E_1(\FF_q) \cong \ZZ_2  \, \oplus \, \ZZ_{\frac{q+1}{2}} $,
        there exist points $ P $ and $ Q $ in $ E_1(\FF_q) $ such that $ 2 P = O $ and $ m Q = O $, where $O$ is the point at infinity. The three distinct points $ A_i = \alpha_i P + \beta_i Q  $, with
        $ \alpha_i \in \ZZ_2, \beta_i \in \ZZ_m $, for $ i=1,2,3$, are collinear if and only if 
        \begin{align}
        & \sum_{i=1}^{3}	\alpha_i P + \beta_i Q = O \iff   \alpha_1 + \alpha_2 + \alpha_3 = 0 \mod 2 , \nonumber \\ &\text{  and  } \beta_1 + \beta_2 + \beta_3 = 0 \mod m. \label{eq:noncyclic}
        \end{align}
Therefore, finding the number of $ 3 $-rich lines is equivalent to finding the number of distinct solutions (up to ordering) to the equation $ x + y + z =0  $ in $ \ZZ_2 \, \oplus \, \ZZ_m $. Now the result follows from 
        Corollary~\ref{Cor:n1n2}, with $ n_1 =2 $ and $ n_2 = m = \frac{q+1}{2}$.
        
        Next we consider $ q+1  $ points on $ E_2(\FF_q)  \cong \, \ZZ_{q+1}$. Let $ P $ be a generator of the underlying group of $ E_2(\FF_q) $. Let points $ P_i \in E_2(\FF_q) $, for $ i=1,2,3 $. Then $ P_1 = xP $, $ P_2 = yP $ and $ P_3 =z P $ for some $ x,y,z \in \ZZ_{q+1} $. Clearly, $ P_1 $, $ P_2 $ and $ P_3 $ are collinear if and only if 
        \begin{equation}
        P_1 + P_2 +P_3 = O \iff xP + yP + z P = O \iff x + y +z = 0 \mod q+1.
        \end{equation}
        Therefore, the number of $ 3 $-rich lines is equal to the number of solutions of the equation $ x+y+z =0  $ in $ \ZZ_{q+1} $. Now, the result follows from Corollary~\ref{Cor:N},
        with $ N =q+1 $.
         \item Let $ q =p^n $ with $ n $ odd. Then from Theorem~\ref{Theorem:Schoof} there exist an elliptic curve such that $ \#E(\FF_q) = q+1 $, i.e., $ t =0 $ in Theorem~\ref{Theorem:Schoof}. Then Theorem~\ref{Theorem:SchoofSupersingular} implies that $ F_{\FF_q} \cong \ZZ_{q+1} $, as $ q \not \equiv 3 \mod 4 $. On the other hand, if $ q =p^n  $ with $ n $ even, then from Theorem~\ref{Theorem:Schoof} it follows that there exists an elliptic curve such that $ \#E(\FF_q) = q+1 $ if and only if $ p \equiv 1 \mod 4 $. Once again, Theorem~\ref{Theorem:SchoofSupersingular} implies that $ F_{\FF_q} \cong \ZZ_{q+1} $. Now the proof proceeds similar to the proof of part (1).
	\end{enumerate}

\end{proof}

\begin{theorem} \label{Theorem:qeven}
	Assume $ \FF_q $ to be a finite field with $ q = 2^n $.
	\begin{enumerate}
		\item Let $ q = 2^n $ with $ n $ odd.  There exits a point-line arrangement $ (N, \floor {(N) (N-3)/6} + 1) $  in the projective plane over the finite field $ \FF_q$ with the group model $ \ZZ_{N} $ for each of the following values of $ N $  
		\begin{enumerate}
	\item	$ N = q+1 $, 
	\item   $ N = q+1 + \sqrt{2q} $,
	\item $ N =  q+1 - \sqrt{2q} $ .
		\end{enumerate}

	\item Let $ q = 2^n $ with $ n $ even. There exits a point-line arrangement $ (N, \floor {(N) (N-3)/6} + 1) $  in the projective plane over the finite field $ \FF_q$ with the group models $ \ZZ_{N} $ for $ N =q+1 $ and for $ N = q +1 \pm \sqrt{q} $.
			
For both the cases (1) and (2),
\begin{equation}
\Or {N} \geq \floor{\frac{N (N-3)}{6}} + 1,
\end{equation}
and it agrees with the Green-Tao bound. 
	\end{enumerate}

\end{theorem}

\begin{proof}
	From Table~\ref{Table:odd} one can pick the elliptic curves $ y^2 + y = x^3 $, $ y^2 + y = x^3 + x $ and $ y^2 + y = x^3 + x +1 $  to get the desired group models. Then the number of $ 3 $-rich lines can be calculated using Corollaries to Lemma~\ref{Lemma:MainGroupSolutionCount}, similar to the calculation carried out in  Theorem~\ref{Theorem:qodd}. This proves part (1) of the theorem. Similarly, part(2) of the theorem follows from Table~\ref{Table:even}.
\end{proof}
\newpage
\begin{theorem} \label{Theorem:qevenGreenTaoExceed}
	 Let $ N = q+ 1 + 2 \sqrt{q} $ with either $ q $ odd or $ q = 2^n $ with $ n $ even. There exist point-line arrangements $ (N, t) $  in the projective plane over the finite field $ \FF_q$ with 
		\begin{enumerate}
			\item the group model $ \ZZ_{\sqrt{q}-1} \, \oplus  \ZZ_{\sqrt{q}-1} $ for $ N = q +1 - 2 \sqrt{q} $, and 
			\begin{align*}
			t = 
			\begin{cases}
			\frac{1}{6} \left( q^2 - 4 q^{\frac{3}{2}} +  3q + 2 \sqrt{q} +16 \right), \quad &\text{if }\, \sqrt{q} \equiv 1 \mod 3, \\
			\frac{1}{6} \left( q^2 - 4 q^{\frac{3}{2}} +  3q + 2 \sqrt{q} \right), \quad   &\text{if }\, \sqrt{q} \not \equiv 1 \mod 3.
			\end{cases}
			\end{align*}
			
			\item the group model $ \ZZ_{\sqrt{q}+1} \, \oplus  \ZZ_{\sqrt{q}+1} $ for $ N = q +1 + 2 \sqrt{q} $, and 
			\begin{align*}
			t = 
			\begin{cases}
			\frac{1}{6} \left( q^2 + 4 q^{\frac{3}{2}} +  3q - 2 \sqrt{q} +16 \right), \quad &\text{if }\, \sqrt{q} \equiv 1 \mod 3, \\
			\frac{1}{6} \left( q^2 + 4 q^{\frac{3}{2}} +  3q - 2 \sqrt{q} \right), \quad   &\text{if }\, \sqrt{q} \not \equiv 1 \mod 3,
			\end{cases}
			\end{align*}
		\end{enumerate}
 with 
		
			\begin{align}
		\Or {N} \geq
			\begin{cases}
				 \floor {\frac{N (N-3)}{6}} + 3 \qquad  &\text{if }\,  3 | N ,  \, \,\text{ (Green-Tao bound is exceeded.)}\\
				 \floor {\frac{N (N-3)}{6}} + 1 \qquad  &\text{if }\,  3 \nmid N.
			\end{cases}
			\end{align}

\end{theorem}

\begin{proof}
	The proof is very similar to the previous Theorem and follows from Table~\ref{Table:odd} and Table~\ref{Table:even}. We omit the details. 
\end{proof}

\begin{theorem} \label{Theorem:Classification}
	Assume $ |t| \leq 2 \sqrt{p}$. Let $ N = p + 1 + t $ such that $ N = n_1 n_2 $, with $ n_1 | n_2 $.  Then there exists an elliptic curve $ E $ over the finite field $ \FF_p $, such that $ E(\FF_p) = N $. Then there exists a point-line arrangement $ (N, t) $  in the projective plane over the finite field $ \FF_p$ with the group model $ \ZZ_{n_1} \, \oplus  \ZZ_{n_2} $, with  
	\begin{align}
	\Or {N} \geq
	\begin{cases}
	\floor {\frac{N (N-3)}{6}} + 3 \qquad  &\text{if }\,  3 | n_1 , \, \,\text{ (Green-Tao bound is exceeded.)}\\
	\floor {\frac{N (N-3)}{6}} + 1 \qquad  &\text{if }\,  3 \nmid n_1.  \\
		\end{cases}
	\end{align}
\end{theorem}
\begin{proof}
There exists an elliptic curve with $ \#E(\FF_p) = N $ for each integer $ N $ in the interval $ [p+1 - 2 \sqrt{p}, p+1 + \sqrt{p}] $  is well-known (for example, it follows from the work of Deuring \cite{MR5125}.) The statement about the group-structure of $ E(\FF_p) $ follows form Theorem~\ref{Theorem:Ruck}. It is clear from the proof of Theorem~\ref{Theorem:qodd} that the number of $ 3 $-rich lines depends upon the underlying group-structure of $ E(\FF_p) $, and more precisely, the number of distinct solutions to the equation $ x+ y+z =0  $ 
in the underlying group of $ E(\FF_p) $. Now, the result follows from the Corollaries of Lemma~\ref{Lemma:MainGroupSolutionCount}.  

\end{proof}

\begin{Remark}
	Some examples of orchard group models are shown in Table~\ref{Table:examples}.
\end{Remark}


\begin{table}[htbp]
	\centering
	\begin{tabular}{l*{5}{S[table-format = 2.1]}}
		\toprule
		\textbf{Curve}                                    &	
		\textbf{q} &
		{\specialcellbold{Group   }}                 &
		{\specialcellbold{N}}         & 
		{\specialcellbold{t}}     & 
		 {\specialcellbold{ Green-Tao \\ bound}}  \\
		\midrule
$ y^2 + y = x^3 + x$       &  8   & \text{$ \ZZ_5  $} & 5  & 2 & 2\\
$ y^2 + y = x^3 + x$       &  128   & \text{$ \ZZ_{145}  $} & 145  & 3432 & 3432\\
$y^2 = x^3 + 1 $       & 5 &  \text{$ \ZZ_{6}$}    &   6  & 4  &  4 \\
$y^2 = x^3 + 1 $       & 49 &  \text{$ \ZZ_{48}$}    &   48  & 361 &  361 \\
$y^2 = x^3 + x $       & 7 &  \text{$ \ZZ_{8}$}    &   8  & 7 &  7 \\
$y^2 = x^3 + x $       & 13 &  \text{$ \ZZ_{2} \oplus \ZZ_{10}$}    &   20  & 57 &  57 \\
$y^2 + y = x^3 - x^2 - 10x - 20 $       & 19 &  \text{$ \ZZ_{20} $}    &   20  & 57 &  57 \\
$ y^2 + y = x^3 $       &  4   & \text{$ \ZZ_3 \oplus \ZZ_3 $} & 9  & 12  & 10\\
$ y^2 + y = x^3 $       & 16  &  \text{$ \ZZ_9 \oplus \ZZ_9 $}   &   81  & 1056  &  1054 \\
$ y^2 + y = x^3 $       & 256 &  \text{$ \ZZ_{15} \oplus \ZZ_{15} $}   &   225  & 8328  &  8326 \\
$y^2 = x^3 + 1 $       & 25 &  \text{$ \ZZ_6 \oplus \ZZ_6 $}    &   36  & 201 &  199 \\
$y^2 = x^3 + 1 $       & 7 &  \text{$ \ZZ_2 \oplus \ZZ_6 $}    &   12  & 19 &  19 \\
				\bottomrule
	\end{tabular}
	\caption{Examples.} 	\label{Table:examples}
\end{table}

\section{Realization on a real projective plane} \label{Sect:Examples}
Let us consider the example of the elliptic curve  $E: \, y^2 = x^3 + 5x^2 + 4x  $ over $ \FF_7 $. It can be checked that
$ \#E(\FF_7) = 8 $ and $ E(\FF_7) $ contains the following $ 8 $ points:
\begin{align*}
A &= (0, 0), \quad B = (2, 1),\quad C = (2, 6), \quad D = (3, 0),\\
E &= (5, 2), \quad  F = (5,5), \quad G= (6,0), \quad H = (0,1,0)  \text{ (the point at infinity).}
\end{align*}

The following seven $ 3 $-rich lines could be formed using the points in $ E(\FF_7) $ (see Fig.~\ref{Fig:8points7linesnonst}): 
\begin{align*}
\l_1: \quad   \text{containing} \quad  \{A, D, G\}, &\qquad \l_2: \quad   \text{containing} \quad  \{O, B, C\},\\
\l_3: \quad   \text{containing} \quad  \{O, E, F\}, & \qquad \l_4: \quad   \text{containing} \quad  \{B, D, F\},\\
\l_5: \quad   \text{containing} \quad  \{B, E, G\}, & \qquad \l_6: \quad   \text{containing} \quad  \{C, D, E\}, \\
\l_7: \quad   \text{containing} \quad  \{C, F, G\}. &
\end{align*}
It can be verified that $ E(\FF_7) \cong \ZZ_4 \times \ZZ_2 $ and the generators of the group are $ C =(2,6) $ and $ D = (3,0) $. Thus, we have a model showing that 
$ \Or{8} \geq 7 $. The orchard configuration with (8 points, 7 lines) in a projective plane over $ \FF_7 $ is shown in Fig.~\ref{Fig:8points7linesnonst}. 
\begin{figure}[ht]
	\includegraphics[]{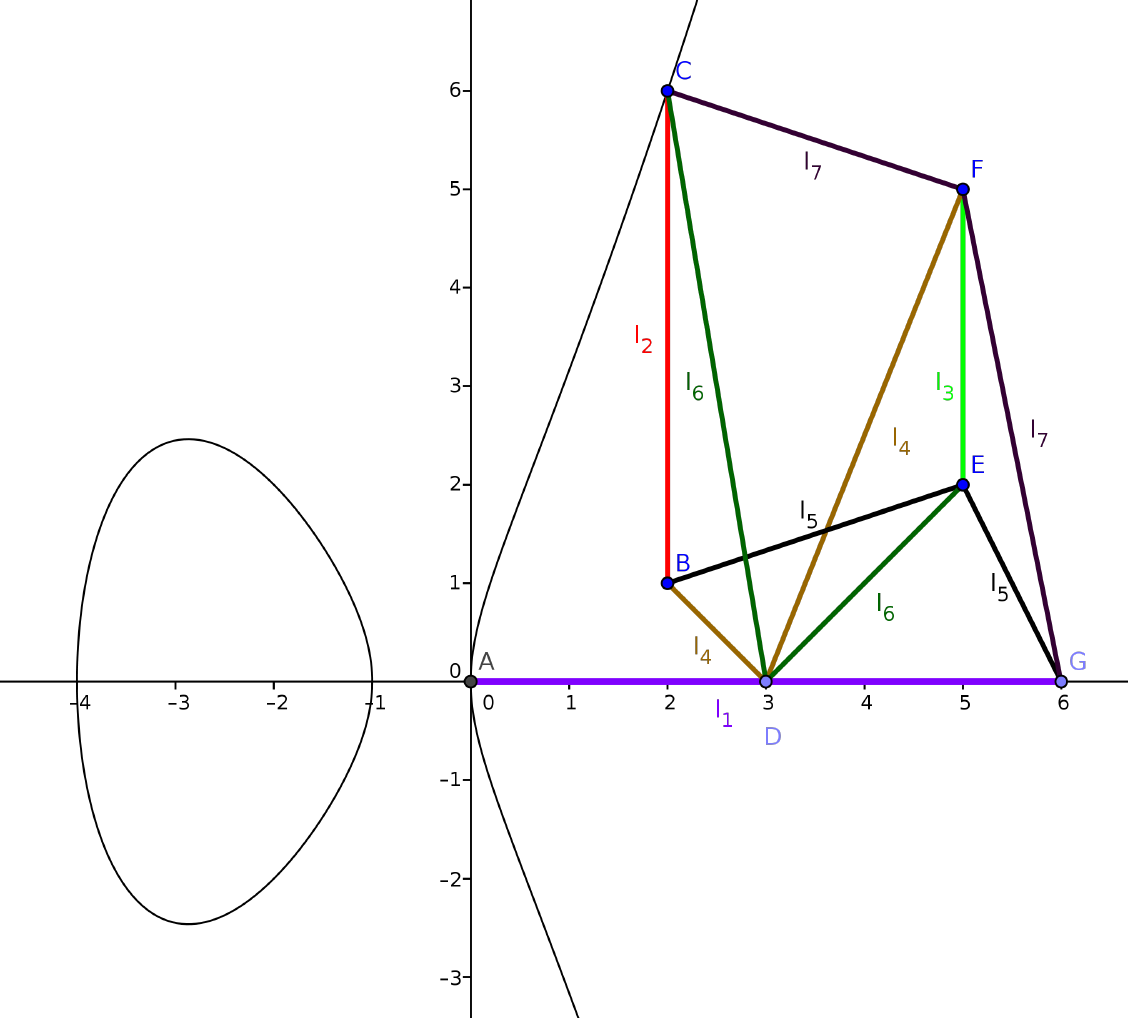}
	\caption{Orchard configuration with ($ 8 $ points, $ 7 $ lines) using elliptic curve given by a minimal Weierstrass equation $ y^2 =  x^3 + 5x^2 + 4x $  over $ \FF_7.$ The point at infinity $ (0,1,0) $ is not shown.}\label{Fig:8points7linesnonst}
\end{figure}
Some lines, for example, $ l_4 $ going through points $B $, $ D $ and $ F $, `visually' does not appear to be a straight line. Is it possible to `straightened' these lines by picking some other equivalent points in the projective plane over $ \FF_7 $, i.e., can this point-line configuration be embedded into a real projective plane to give an orchard model? The answer in this case is in affirmative, as shown in Fig.~\ref{Fig:8points7linesSt}. 

\begin{figure}[]
	\includegraphics[]{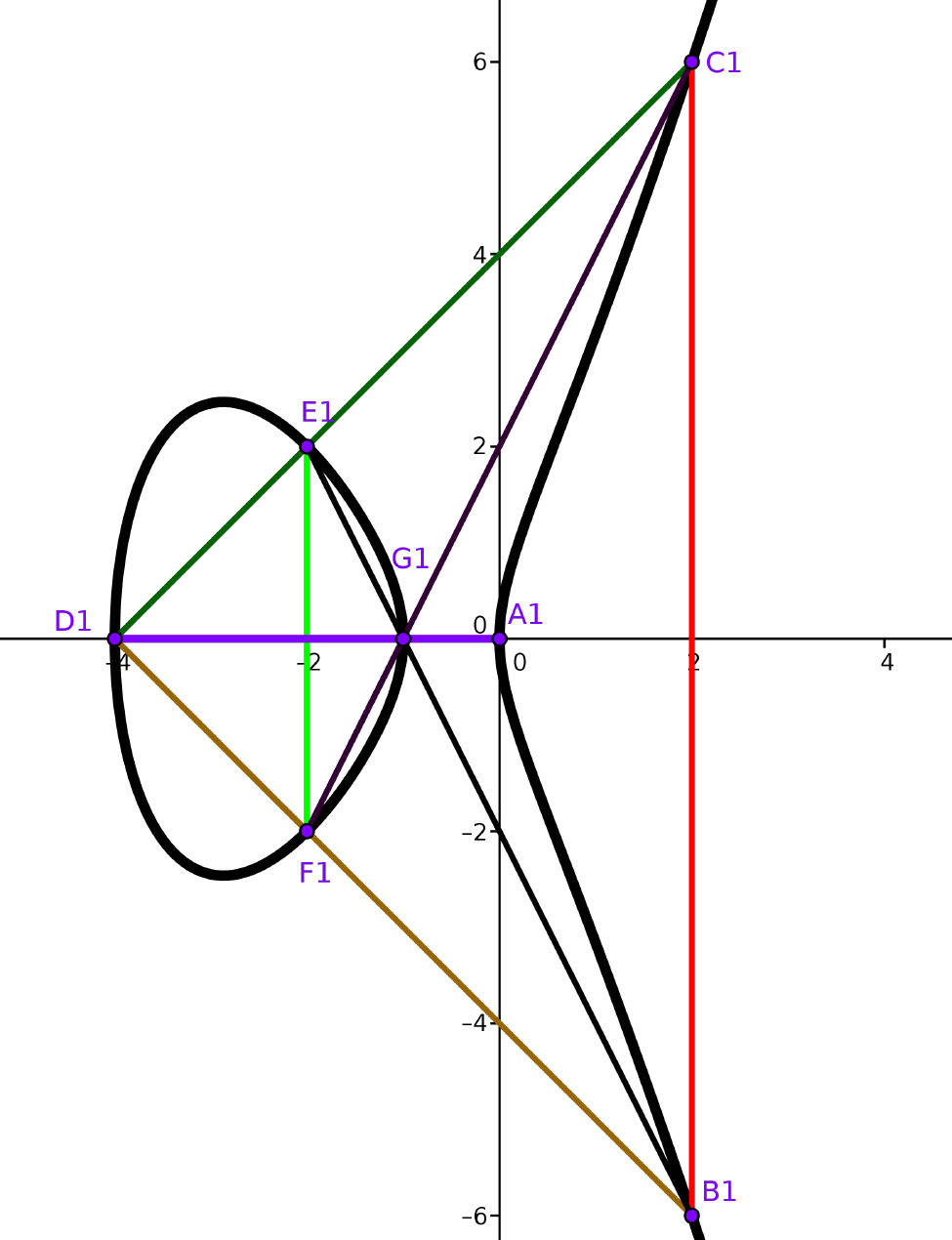}
	\caption{Orchard configuration with ($ 8 $ points, $ 7 $ lines) using elliptic curve given by a minimal Weierstrass equation $ y^2 =  x^3 + 5x^2 + 4x $  on a projective plane over $ \RR $. The point at infinity $ (0,1,0) $ is not shown.}\label{Fig:8points7linesSt}
\end{figure}

Note that 
\begin{align*}
A_1 &=A  = (0,0), \quad B_1 = (2,-6) \equiv B = (2,1) \mod  7, \quad C_1 = C = (2,6) \\
D_1 &= (-4,0) \equiv D = (3,0) \mod 7, \quad E_1 = (-2,2) \equiv (5,2) \mod 7 \\
F_1 & = (-2,-2) \equiv F = (5,5,) \mod 7, \quad G_1  = (-1,0) \equiv G = (6,0) \mod 7.
\end{align*}

In Fig.~\ref{Fig:8points7linesSt}  seven points $ A_1 $, $ B_1 $, $ C_1 $, $ D_1 $, $ E_1 $, $ F_1 $ and $ G_1 $ are arranged so that they yield total  $ 7 $ `straight' $ 3 $-rich lines.
In fact, the configuration shown in Fig.~\ref{Fig:8points7linesSt} gives an orchard model  with ($ 8 $ points, $ 7 $ lines) for $ y^2 = x^3 + 5x^2 + 4x  $ over any finite $ \FF_p $, for all prime $ p \geq 7 $. Moreover, it is easy to check that the above points on $ y^2 = x^3 + 5x^2 + 4x  $  also give an orchard model on a real projective plane. This is nice, as one obtains a classical orchard model ($ 8 $ points, $ 7 $ lines) with integral points (i.e., points with integral coordinates) and the group structure $ \ZZ_4 \times \ZZ_2 $ over real projective plane.

A natural question arises: is it always possible to obtain an orchard group model on a real projective plane corresponding to an orchard model on a projective plane over a finite field?
Of course, for orchard group model with ($n$  points, $ t $ lines), such that $ t =  \floor {\frac{n (n-3)}{6}} + 3$, i.e., where the Green-Tao bound is exceeded, the answer to this question would be no for sufficiently large $ n $, otherwise it will contradict Theorem~\ref{Theorem:GreenTao}. Indeed, the answer is no for even a small $ n = 9$, as evident from the following example.

Consider the elliptic Curve defined by $ y^2 = x^3 + 2 $ over $ \FF_7 $.
It can be checked that
$ \#E(\FF_7) = 9 $ and $ E(\FF_7) $ contains the following $ 9 $ points:
\begin{align*}
A &= (0, 3), \quad B = (0, 4),\quad C = (3, 1), \quad D = (3, 6), \quad E = (5,1), \\
F &= (5,6), \quad G = (6,1), \quad H= (6,6), \quad O = (0,1,0)  \text{ (the point at infinity).}
\end{align*}
The following twelve $ 3 $-rich lines could be formed using the points in $ E(\FF_7) $ 
\begin{align*}
\l_1: \quad   \text{containing} \quad  \{O, A, B\}, &\qquad \l_2: \quad   \text{containing} \quad  \{O, C, D\},\\
\l_3: \quad   \text{containing} \quad  \{O, E, F\}, & \qquad \l_4: \quad   \text{containing} \quad  \{O, G, H\},\\
\l_5: \quad   \text{containing} \quad  \{A, C, H\}, & \qquad \l_6: \quad   \text{containing} \quad  \{A, D, E\}, \\
\l_7: \quad   \text{containing} \quad  \{A, F, G\}. & \qquad \l_8: \quad   \text{containing} \quad  \{B, C, F\}, \\
\l_9: \quad   \text{containing} \quad  \{B, D, G\}, &\qquad \l_{10}: \quad   \text{containing} \quad  \{B, E, H\},\\
\l_{11}: \quad   \text{containing} \quad  \{C, E, G\}, & \qquad \l_{12}: \quad   \text{containing} \quad  \{D, F, H\}.
\end{align*}

It can be verified that $ E(\FF_7) \cong \ZZ_3 \times \ZZ_3 $ and the generators of the group are $ C =(6,6) $ and $ D = (6,0) $. Thus, we have a model showing that 
$ \Or{9} \geq 12 $. The orchard configuration with ($ 9 $ points, $ 12 $ lines) in a projective plane over $ \FF_7 $. But, this configuration can not be realized over the real 
projective plane as it is known that an optimal configuration of $ 9 $ points can yield at most $ 10 $  $ 3 $-rich lines (see \cite{OEISA003035}).

\bibliographystyle{plain}

\end{document}